\UseRawInputEncoding
\documentclass[12pt]{article}

\usepackage{dsfont}
\usepackage{amsfonts,amsmath,amsthm}
\usepackage{algorithm, algorithmic}
\usepackage{color}
\usepackage{graphicx}
\usepackage{stmaryrd}        

\usepackage[paper=a4paper,dvips,top=2cm,left=2cm,right=2cm,
    foot=1cm,bottom=4cm]{geometry}

\begin{document}

\title{Dual Quaternions and Dual Quaternion Vectors}
\author{ Liqun Qi\footnote{%
    Department of Applied Mathematics, The Hong Kong Polytechnic University, Hung Hom,
    Kowloon, Hong Kong;
    Department of Mathematics, School of Science, Hangzhou Dianzi University, Hangzhou 310018 China
    ({\tt maqilq@polyu.edu.hk}).   This author's work was supported by Hong Kong Innovation and Technology Commission (InnoHK Project CIMDA).}
    \and \
    Chen Ling\thanks{Department of Mathematics, Hangzhou Dianzi University, Hangzhou, 310018, China; ({\tt
macling@hdu.edu.cn}).  This author's work was supported by Natural Science Foundation of China (No. 11971138) and Natural Science Foundation of Zhejiang Province (No. LY19A010019, LD19A010002).}
  \and and \
Hong Yan\thanks{Department of Electrical Engineering, City University of Hong Kong, Kowloon, Hong Kong.  This author's work was supported by Hong Kong Innovation and Technology Commission (InnoHK Project CIMDA) and City University of Hong Kong (Project 9610034). }
}
\date{\today}
\maketitle

\begin{abstract}
We introduce a total order and an absolute value function for dual numbers.   The absolute value function of dual numbers takes dual number values, and has properties similar to the those of the absolute value function of real numbers.  
We define the magnitude of a dual quaternion, as a dual number.    Based upon these, we extended $1$-norm, $\infty$-norm and $2$-norm to dual quaternion vectors.

\medskip



\end{abstract}

\renewcommand{\Re}{\mathds{R}}
\newcommand{\rank}{\mathrm{rank}}
\renewcommand{\span}{\mathrm{span}}
\newcommand{\X}{\mathcal{X}}
\newcommand{\A}{\mathcal{A}}
\newcommand{\I}{\mathcal{I}}
\newcommand{\B}{\mathcal{B}}
\newcommand{\C}{\mathcal{C}}
\newcommand{\OO}{\mathcal{O}}
\newcommand{\e}{\mathbf{e}}
\newcommand{\0}{\mathbf{0}}
\newcommand{\dd}{\mathbf{d}}
\newcommand{\ii}{\mathbf{i}}
\newcommand{\jj}{\mathbf{j}}
\newcommand{\kk}{\mathbf{k}}
\newcommand{\va}{\mathbf{a}}
\newcommand{\vb}{\mathbf{b}}
\newcommand{\vc}{\mathbf{c}}
\newcommand{\vg}{\mathbf{g}}
\newcommand{\vr}{\mathbf{r}}
\newcommand{\vt}{\rm{vec}}
\newcommand{\vx}{\mathbf{x}}
\newcommand{\vy}{\mathbf{y}}
\newcommand{\vu}{\mathbf{u}}
\newcommand{\vv}{\mathbf{v}}
\newcommand{\y}{\mathbf{y}}
\newcommand{\vz}{\mathbf{z}}
\newcommand{\T}{\top}

\newtheorem{Thm}{Theorem}[section]
\newtheorem{Def}[Thm]{Definition}
\newtheorem{Ass}[Thm]{Assumption}
\newtheorem{Lem}[Thm]{Lemma}
\newtheorem{Prop}[Thm]{Proposition}
\newtheorem{Cor}[Thm]{Corollary}
\newtheorem{example}[Thm]{Example}

\section{Introduction}

Quaternions were introduced by Hamilton in 1843 \cite{Ha43}.   In 1873, Clifford \cite{Cl73} introduced dual numbers, dual complex numbers and dual quaternions.   This results in a new branch of algebra - geometric algebra or Clifford algebra.    Now, dual numbers, dual complex numbers and dual quaternions have found wide applications in automatic differentiation, geometry, mechanics, rigid body motions, robotics and computer graphics \cite{BK20, BLH19, CKJC16, Da99, Gu11, MKO14, WYL12}.

However, there are some gaps between the applications of dual quaternions and their mathematical foundations.   For example, unit dual quaternions play an important role to represent the motion of a rigid body in 3D \cite{BK20, CKJC16, Ke12}.  This involves the definition of the magnitude or norm of a dual quaternion.  In \cite{Ke12}, the magnitude of a dual quaternion is defined as
$$\|q\| = qq^*,$$
where $q^*$ is the conjugate of $q$.  In the above definition, the left side has only one $q$ factor, while the right side has two $q$ factors.  This is not consistent.  It is better to change it to
$$\|q\|^2 = qq^*.$$
Then
$$\|q\| = \sqrt{qq^*}.$$
This is not well-defined if $q$ is infinitesimal.  Also, now, $\|q\|$ is a dual number.  Does $\|\cdot \|$ still obeys the triangular inequality?   Some investigations are needed.

In the next section, we introduce a total order for dual numbers.   We also define the square root function for dual numbers.

Then, in Section 3, we define the absolute value function for dual numbers.   We see that it inherits many properties of the absolute value function of real numbers.

We show in Section 4 that  the sum of the product of a quaternion and the conjugate of another quaternion, and the product of the other quaternion and the conjugate of that quaternion, is a real number.

We define in Section 5 the magnitude of a dual quaternion, as a dual number.    This definition matches the definition of unit dual quaternions in applications.

Based upon these, in Section 6, we extend $1$-norm, $\infty$-norm and $2$-norm to dual quaternion vectors.  The first two extensions are direct, while the third extension is nontrivial.

Some final remarks are made in Section 7.

We denote scalars, vectors and matrices by small letters, bold small letters and capital letters, respectively.

\section{A Total Order for Dual Numbers}

Denote $\mathbb R$ and $\mathbb D$ as the set of the real numbers, and the set of the dual numbers, respectively.   A dual number $q$ has the form $q = q_{st} + q_\I\epsilon$, where $q_{st}$ and $q_\I$ are real numbers,  and $\epsilon$ is the infinitesimal unit, satisfying $\epsilon^2 = 0$.   We call $q_{st}$ the real part or the standard part of $q$, and $q_\I$ the dual part or the infinitesimal part of $q$.  The infinitesimal unit $\epsilon$ is commutative in multiplication with real numbers, complex numbers and quaternion numbers.  The dual numbers form a commutative algebra of dimension two over the reals.    If $q_{st} \not = 0$, we say that $q$ is appreciable, otherwise, we say that $q$ is infinitesimal.

We may define a total order $\le$ over $\mathbb D$.   Suppose $p = p_{st} + p_\I\epsilon, q = q_{st} + q_\I\epsilon \in \mathbb D$.  We have $q < p$ if $q_{st} < p_{st}$, or $q_{st} = p_{st}$ and $q_\I < p_\I$.  We have $q = p$ if $q_{st} = p_{st}$ and $q_\I = p_\I$.   Thus, if $q > 0$, we say that $q$ is a positive dual number; and if $q \ge 0$, we say that $q$ is a nonnegative dual number.  Denote the set of nonnegative dual numbers by $\mathbb D_+$, and the set of positive dual numbers by $\mathbb D_{++}$.

For $p = p_{st} + p_\I\epsilon, q = q_{st} + q_\I\epsilon \in \mathbb D$ and a positive integer $k$, we have
\begin{equation} \label{e1}
p + q = p_{st}+q_{st} +(p_\I+q_\I)\epsilon,
\end{equation}
\begin{equation} \label{e2}
pq = p_{st}q_{st} +(p_{st}q_\I+p_\I q_{st})\epsilon,
\end{equation}
\begin{equation} \label{e3}
q^k = q_{st}^k + kq_{st}^{k-1}q_\I \epsilon.
\end{equation}

Then we have the following theorem.

\begin{Thm}
For any $p, q \in \mathbb D$ and a positive integer $k$, we have the following conclusions.

1. $q^{2k} \in \mathbb D_+$;

2. $p^2 + q^2 - 2pq \in \mathbb D_+$;

3. If $p, q \in \mathbb D_+$, then $pq \in \mathbb D_+$;

4. If $p, q \in \mathbb D_{++}$ and at least one of them is appreciable, then $pq \in \mathbb D_{++}$
\end{Thm}
\begin{proof}  1. By (\ref{e3}), we have
$$q^{2k} = q_{st}^{2k} + 2kq_{st}^{2k-1}q_\I \epsilon.$$
If $q_{st} \not = 0$, then $q_{st}^{2k} > 0$.   This implies $q^{2k} > 0$.   If $q_{st} = 0$, then $q^{2k} = 0$.

2.  By 1, we have $p^2 + q^2 - 2pq = (p-q)^2 \ge 0$.

3. If $p_{st} > 0$ and $q_{st} > 0$, then $p_{st}q_{st} > 0$.   By (\ref{e2}), $pq > 0$.  If $p_{st} = 0$ and $q_{st} > 0$, then $p_\I \ge 0$ as $p \ge 0$.  By (\ref{e2}), $pq = p_\I q_{st}\epsilon \ge 0$.   Similarly, if $p_{st} > 0$ and $q_{st} = 0$, then by (\ref{e2}), $pq = p_{st}q_\I \epsilon \ge 0$. If $p_{st} = q_{st} = 0$, then by (\ref{e2}), $pq = 0$.

4. This may be proved similarly to 3.

\end{proof}

Clearly, many inequalities of real numbers can be extended to dual numbers without difficulty.

For $p, q \in \mathbb D$, suppose that $p \le q$.  Then we may define
$$[p, q] = \{ u \in \mathbb D : p \le u \le q \},$$
$$[p, +\infty) = \{ u \in \mathbb D : p \le u \},$$
$$(p, +\infty) = \{ u \in \mathbb D : p < u \},$$
$$(-\infty, q] = \{ u \in \mathbb D : u \le q \},$$
$$(-\infty, q) = \{ u \in \mathbb D : u < q \}.$$
If furthermore $p < q$, then we may define
$$[p, q) = \{ u \in \mathbb D : p \le u < q \},$$
$$(p, q] = \{ u \in \mathbb D : p < u \le q \},$$
$$(p, q) = \{ u \in \mathbb D : p < u < q \}.$$

If $q$ is appreciable, then $q$ is invertible and
$$q^{-1} = q_{st}^{-1} - q_{st}^{-1}q_\I q_{st}^{-1}\epsilon.$$
If $q$ is infinitesimal, then $q$ is not invertible.

If $q$ is nonnegative and appreciable, then the square root of $q$ is still a nonnegative dual number.   If $q$ is positive and appreciable, we have
\begin{equation} \label{e4}
\sqrt{q} = \sqrt{q_{st}} + {q_\I \over 2\sqrt{q_{st}}}\epsilon.
\end{equation}
When $q=0$, we have $\sqrt{q} = 0$.

\section{The Absolute Value Function of Dual Numbers}

Recall that for any $u \in \mathbb R$,
$${\rm sgn}(u) =  \left\{ \begin{aligned} 1, & \ {\rm if}\  u > 0, \\ 0, &   \ {\rm if}\  u = 0, \\
-1, &   \ {\rm if}\  u < 0.  \end{aligned}  \right. $$

We define the absolute value of $q \in \mathbb D$ as
\begin{equation} \label{e5}
|q| = \left\{ \begin{aligned} |q_{st}| + {\rm sgn}(q_{st})q_\I\epsilon, & \ {\rm if}\  q_{st} \not = 0, \\ |q_\I|\epsilon, &   \ {\rm otherwise}.  \end{aligned}  \right.
\end{equation}

We have the following theorem.

\begin{Thm}
The mapping $|\cdot| : \mathbb D \to \mathbb D_+$.   Suppose that $p, q \in \mathbb D$.   Then,

1. $|q| = 0$ if and only if $q = 0$;

2. $|q| = q$ if $q \ge 0$, $|q| > q$ otherwise;

3. $|q| = \sqrt{q^2}$ if $q$ is appreciable;

4. $|pq| = |p|  |q|$;

5. $|p+q| \le |p| + |q|$.

\end{Thm}
\begin{proof}
By definition, $|\cdot| : \mathbb D \to \mathbb D_+$, and  we may verify conclusions 1 and 2 directly.

Suppose that $q$ is appreciable.  Then $q_{st} \not = 0$.  We have $q^2 = q_{st}^2 + 2q_{st}q_\I \epsilon$.    This implies that
$$\sqrt{q^2} = \sqrt{q_{st}^2} + {2q_{st}q_\I \over 2\sqrt{q_{st}^2}}\epsilon = |q_{st}| + {\rm sgn}(q_{st})q_\I\epsilon = |q|.$$
We have conclusion 3.

We have
$$pq = p_{st}q_{st} + (p_{st}q_\I + p_\I q_{st})\epsilon.$$
Then,
$$|pq| = \left\{ \begin{aligned} |p_{st}q_{st}| + {\rm sgn}(p_{st}q_{st})(p_{st}q_\I + p_\I q_{st})\epsilon, & \ {\rm if}\  p_{st}q_{st} \not = 0, \\ |p_{st}q_\I + p_\I q_{st}|\epsilon, &   \ {\rm otherwise}.  \end{aligned}  \right. $$

If $p_{st} \not = 0$ and $q_{st} \not = 0$, then
$${\rm sgn}(p_{st}q_{st}) = {\rm sgn}(p_{st}){\rm sgn}(q_{st}).$$
We have
\begin{eqnarray*}
|pq| & = & |p_{st}q_{st}| + {\rm sgn}(p_{st}q_{st})(p_{st}q_\I + p_\I q_{st})\epsilon \\
& = & |p_{st}| |q_{st}| + |p_{st}|{\rm sgn}(q_{st})q_\I \epsilon + {\rm sgn}(p_{st})p_\I |q_{st}|\epsilon \\
& = & \left(|p_{st}|+ {\rm sgn}(p_{st})p_\I \epsilon\right)\left(|q_{st}|+ {\rm sgn}(q_{st})q_\I \epsilon\right)\\
& = & |p| |q|.
\end{eqnarray*}

If $p_{st} = 0$ and $q_{st} \not = 0$, then $pq = p_\I q_{st}\epsilon$.   This implies that $|pq| = |p_\I | |q_{st}|\epsilon$.  On the other hand, we have $|p| = |p_\I |\epsilon$ and $|q| = |q_{st}| + {\rm sgn}(q_{st})q_\I \epsilon$.   Therefore,
$$|p||q|= |p_\I| |q_{st}|\epsilon = |pq|.$$

Similarly, if $p_{st} \not = 0$ and $q_{st} = 0$, then we have $|p| |q| = |pq|$.

If $p_{st} = q_{st} = 0$, then $pq = 0$, $|p| = |p_\I |\epsilon$ and $|q| = |q_\I |\epsilon$.  We still have $|pq|= |p| |q|$.    Hence, conclusion 4 holds.

We now show conclusion 5.  We have
$$|p + q| = |(p_{st}+p_\I \epsilon) + (q_{st}+ q_\I \epsilon)|= \left\{ \begin{aligned} |p_{st} + q_{st}| + {\rm sgn}(p_{st}+q_{st})(p_\I + q_\I)\epsilon, & \ {\rm if}\  p_{st} + q_{st} \not = 0, \\
|p_\I + q_\I|\epsilon, &  \ {\rm otherwise}.
\end{aligned} \right.$$

If $p_{st} \not = 0$ and $q_{st} \not = 0$, then
$$|p| + |q|= |p_{st}|+ |q_{st}| + ({\rm sgn}(p_{st})p_\I + {\rm sgn}(q_{st})q_\I) \epsilon.$$
We have $|p+q| \le |p|+|q|$, as $|p_{st}|+ |q_{st}| > 0$ and $|p_{st}+q_{st}| \le |p_{st}|+ |q_{st}|$.

If $p_{st} = 0$ and $q_{st} \not = 0$, then
$$|p| + |q|= |q_{st}| + (|p_\I| + {\rm sgn}(q_{st})q_\I) \epsilon.$$
We have
$$|p + q| = |q_{st}| + {\rm sgn}(q_{st})(p_\I + q_\I) \epsilon \le |p| + |q|.$$

Similarly, if $p_{st} \not = 0$ and $q_{st} = 0$, then we have $|p+q| \le |p|+|q|$.

If $p_{st} = q_{st} = 0$, then
$$|p + q| =  |p_\I + q_\I| \epsilon \le (|p_\I| + |q_\I|) \epsilon = |p| + |q|.$$

Thus, in any case, we have $|p+q| \le |p| + |q|$.
\end{proof}

We see that the absolute value concept of dual numbers extends the absolute value concept of real numbers.

In fact, for $p, q \in \mathbb D$, $|p-q|$ defines the distance between $p$ and $q$.  This distance satisfies the triangular formula, and provides the basis for dual analysis.

Based on this distance measurement, we can define limits and continuous dual functions.   A function $f : (a, b) \to \mathbb D$, where $a, b \in \mathbb D$, $a < b$, is called a dual function.   Then we see that the properties of dual functions are different from real functions.    Consider $f(x) = x^2 - \epsilon$, defined on $\mathbb D$.   We have $f(0) = -\epsilon < 0$ and $f(1) = 1 - \epsilon > 0$.   However, we cannot find $x \in [0, 1] \subset \mathbb D$ such that $f(x) = 0$.


\section{A Mixed Product Sum of Two Quaternions}

Denote $\mathbb Q$ as the set of the quaternions.
A quaternion $q$ has the form
$q = q_0 + q_1\ii + q_2\jj + q_3\kk,$
where $q_0, q_1, q_2$ and $q_3$ are real numbers, $\ii, \jj$ and $\kk$ are three imaginary units of quaternions, satisfying
$\ii^2 = \jj^2 = \kk^2 =\ii\jj\kk = -1,$
$\ii\jj = -\jj\ii = \kk, \ \jj\kk = - \kk\jj = \ii, \kk\ii = -\ii\kk = \jj.$
The real part of $q$ is Re$(q) = q_0$.   The imaginary part of $q$ is Im$(q) = q_1\ii + q_2\jj +q_3\kk$.
A quaternion is called imaginary if its real part is zero.
The multiplication of quaternions satisfies the distribution law, but is noncommutative.

The conjugate of $q = q_0 + q_1\ii + q_2\jj + q_3\kk$ is
$q^* = q_0 - q_1\ii - q_2\jj - q_3\kk.$
The magnitude of $q$ is
$|q| = \sqrt{q_0^2+q_1^2+q_2^2+q_3^2}.$
It follows that the inverse of a nonzero quaternion $q$ is given by
$q^{-1} = {q^* / |q|^2}.$ For any two quaternions $p$ and $q$, we have $(pq)^* = q^*p^*$.

By Theorem 2.1 of \cite{Zh97}, we have the following proposition.

\begin{Prop} \label{p4.1}
For any $p = p_0 + p_1\ii + p_2\jj + p_3\kk, q = q_0 + q_1\ii + q_2\jj + q_3\kk \in \mathbb Q$, we have

1. $|q| = |q^*|$;

2. $q^*q = qq^* = |q|^2 = q_0^2+q_1^2+q_2^2+q_3^2$;

3. $|q| = 0$ if and only if $q = 0$;

4. $|p+q| \le |p|+|q|$;

5. $|pq| = |p| |q|$.
\end{Prop}

By direct calculation, we have the following theorem.

\begin{Thm} \label{t4.2}
Suppose that $p = p_0 + p_1\ii + p_2\jj + p_3\kk, q = q_0 + q_1\ii + q_2\jj + q_3\kk \in \mathbb Q$, where
$p_0, p_1, p_2, p_3, q_0, q_1, q_2, q_3 \in \mathbb R$.    Then
\begin{equation} \label{e6}
pq^* + qp^* = p^*q + q^*p = 2p_0q_0 + 2p_1q_1 + 2 p_2q_2 + 2p_3q_3,
\end{equation}
which is a real number.
\end{Thm}





\section{The Magnitude of a Dual Quaternion}

We may denote the set of dual quaternions as $\mathbb {DQ}$.   A dual quaternion $q \in \mathbb {DQ}$ has the form
$$q = q_{st} + q_\I\epsilon,$$
where $q_{st}, q_\I \in \mathbb {Q}$ are the standard part and the infinitesimal part of $q$ respectively. The conjugate of $q$ is
$$q^* = q_{st}^* + q_\I^*\epsilon.$$
See \cite{BK20, CKJC16, Ke12}.
If $q_{st} \not = 0$, then we say that $q$ is appreciable.

We can derive that $q$ is invertible if and only if  $q$ is appreciable. In this case, we have
$$q^{-1} = q_{st}^{-1} - q_{st}^{-1}q_\I q_{st}^{-1} \epsilon.$$

The magnitude of $q$ is defined as
\begin{equation} \label{e7}
|q| := \left\{ \begin{aligned} |q_{st}| + {(q_{st}q_\I^*+q_\I q_{st}^*) \over 2|q_{st}|}\epsilon, & \ {\rm if}\  q_{st} \not = 0, \\
|q_\I|\epsilon, &  \ {\rm otherwise},
\end{aligned} \right.
\end{equation}
which is a dual number.    Note that this definition reduces to the definition of the absolute function if $q \in \mathbb D$, and the definition of the magnitude of a quaternion if $q \in \mathbb Q$.

\begin{Thm} \label{t5.1}
The magnitude $|q|$ is a dual number for any $q \in \mathbb {DQ}$.   If $q$ is appreciable, then
\begin{equation} \label{e8}
|q| = \sqrt{qq^*}.
\end{equation}

For any $p, q \in \mathbb {DQ}$, we have

1. $qq^* = q^*q$;

2. $|q| = |q^*|$;

3. $|q|\ge 0$ for all $q$, and $|q| = 0$ if and only if $q = 0$;

4. $|pq| = |p| |q|$;

5. $|p+q| \le |p|+|q|$.

\end{Thm}
\begin{proof}
By Theorem \ref{t4.2}, $q_{st}q_\I^*+q_\I q_{st}^*$ is a real number.  As $|q_{st}|$ is also a real number, $|q|$, defined by (\ref{e7}), is a dual number.   If $q$ is appreciable, then $q_{st} \not = 0$.   We have
$$qq^* = q_{st}q_{st}^* + (q_{st}q_\I^*+q_\I q_{st}^*)\epsilon = |q_{st}|^2 + (q_{st}q_\I^*+q_\I q_{st}^*)\epsilon.$$
By (\ref{e4}) and (\ref{e7}), we have (\ref{e8}).

1.   We have
$$qq^* =q_{st}q_{st}^*+(q_\I q_{st}^* + q_{st}q_\I^*)\epsilon$$
and
$$q^*q =q_{st}^*q_{st}+(q_\I^*q_{st}+q_{st}^*q_\I)\epsilon.$$
Then by Proposition \ref{p4.1} and Theorem \ref{t4.2}, we have $qq^* = q^*q$.

2. If $q$ is appreciable, then by (\ref{e8}), $|q| = \sqrt{qq^*}$ and $|q^*| = \sqrt{q^*q}$.   By 1, we have $|q| = |q^*|$.    If $q$ is not appreciable, then $q = q_\I \epsilon$ and $q^* = q_\I^*\epsilon$.
We have $|q| = |q_\I|\epsilon$ and $|q^*| = |q_\I^*|\epsilon$.   By Proposition \ref{p4.1}, $|q_\I| = |q_\I^*|$.   Thus, we also have $|q| = |q^*|$ in this case.

3. By (\ref{e7}), we have the conclusion.

4. Let $d = pq$.   Denote $p = p_{st} + p_\I\epsilon$, $q = q_{st} + q_\I\epsilon$ and $d = d_{st} + d_\I\epsilon$, where $p_{st}, p_\I, q_{st}, q_\I, d_{st}, d_\I \in \mathbb Q$.     Then
$d_{st} = p_{st}q_{st}$, $d_\I = p_\I q_{st}+ p_{st}q_\I$, and
$$|d| = \left\{ \begin{aligned} |d_{st}| + {(d_{st}d_\I^*+d_\I d_{st}^*) \over 2|d_{st}|}\epsilon, & \ {\rm if}\  d_{st} \not = 0, \\
|q_\I|\epsilon, &  \ {\rm otherwise}.
\end{aligned} \right.$$

If $p_{st} \not = 0$ and $q_{st} \not = 0$, then $d_{st}=p_{st}q_{st} \not = 0$.  We have
$$|p| =  |p_{st}| + {(p_{st}p_\I^*+p_\I p_{st}^*) \over 2|p_{st}|}\epsilon,$$
$$|q| =  |q_{st}| + {(q_{st}q_\I^*+q_\I q_{st}^*) \over 2|q_{st}|}\epsilon,$$
$$|pq| = |d| =  |d_{st}| + {(d_{st}d_\I^*+d_\I d_{st}^*) \over 2|d_{st}|}\epsilon.$$
We have
\begin{eqnarray*}
|p| |q| & = & |p_{st}||q_{st}| +\left({|p_{st}|(q_{st}q_\I^*+q_\I q_{st}^*) \over 2|q_{st}|} +
{|q_{st}|(p_{st}p_\I^*+p_\I p_{st}^*) \over 2|p_{st}|}\right)\epsilon\\
 & = & |p_{st}q_{st}| +{|p_{st}|^2(q_{st}q_\I^*+q_\I q_{st}^*) +
|q_{st}|^2(p_{st}p_\I^*+p_\I p_{st}^*) \over 2|p_{st}| |q_{st}| }\epsilon\\
& = & |d_{st}| +{|p_{st}|^2(q_{st}q_\I^*+q_\I q_{st}^*) +
|q_{st}|^2(p_{st}p_\I^*+p_\I p_{st}^*) \over 2|d_{st}| }\epsilon.
\end{eqnarray*}
Thus, to show that $|pq| = |p||q|$ in this case, it suffices to show that
\begin{equation} \label{e9}
d_{st}d_\I^*+d_\I d_{st}^* = |p_{st}|^2(q_{st}q_\I^*+q_\I q_{st}^*) +
|q_{st}|^2(p_{st}p_\I^*+p_\I p_{st}^*).
\end{equation}
We have $d_{st} = p_{st}q_{st}$ and $d_\I = p_\I q_{st} + p_{st}q_\I$.  Then
$d_{st}^* = q_{st}^*p_{st}^*$ and $d_\I^* = q_{st}^*p_\I^* + q_\I^*p_{st}^*$.
From these, (\ref{e9}) can be derived.  Thus, $|pq| = |p| |q|$ in this case.

If $p_{st} = 0$ and $q_{st} \not = 0$, then $pq = p_\I q_{st}\epsilon$.  This implies that $|pq| = |p_\I||q_{st}|\epsilon$.  On the other hand, we have $p = p_\I \epsilon$, which implies that $|p| = |p_\I |\epsilon$.
Since
$$|q| =  |q_{st}| + {(q_{st}q_\I^*+q_\I q_{st}^*) \over 2|q_{st}|}\epsilon,$$
we have $|p| |q| = |p_\I||q_{st}|\epsilon = |pq|$.

Similarly, if $p_{st} \not = 0$ and $q_{st} = 0$, we also have $|pq| = |p| |q|$.

If $p_{st} = q_{st} = 0$, then $p = p_\I \epsilon$, $|p|= |p_\I| \epsilon$,  $q = q_\I \epsilon$, $|q|= |q_\I| \epsilon$.   We have $|pq| = 0 = |p| |q|$.
Thus, $|pq| = |p| |q|$ in all the cases.  This proves Conclusion 4.

5.  We have
$$|p| =  |p_{st}| + {(p_{st}p_\I^*+p_\I p_{st}^*) \over 2|p_{st}|}\epsilon,$$
$$|q| =  |q_{st}| + {(q_{st}q_\I^*+q_\I q_{st}^*) \over 2|q_{st}|}\epsilon,$$
$$|p+q|  =  |p_{st}+q_{st}| + {(p_{st}+q_{st})(p_\I^*+q_\I^*)+(p_\I+q_\I)(p_{st}^*+q_{st}^*) \over 2|p_{st}+q_{st}|}\epsilon.$$
We may also divide to four cases, namely, a. $p_{st} \not = 0$ and $q_{st} \not = 0$, b. $p_{st} = 0$ and $q_{st} \not = 0$, c. $p_{st} \not = 0$ and $q_{st} = 0$, and d. $p_{st} = q_{st} = 0$, to prove Conclusion 5.   We omit the technical details here.
\end{proof}

If $|q|=1$, then $q$ is called a unit dual quaternion, which plays an important role to represent the motion of a rigid body in 3D \cite{BK20, CKJC16, Ke12}.  Thus, $q$ is a unit dual quaternion if and only if $|q_{st}| = 1$ and
$$q_{st}q_\I^*+q_\I q_{st}^* = 0.$$
This matches the definition of unit dual quaternions in applications \cite{BK20, CKJC16, Ke12}.

\section{Norms of Dual Quaternion Vectors}

Denote the collection of $n$-dimensional quaternion  vectors by ${\mathbb {Q}}^n$, and the collection of $n$-dimensional dual quaternion  vectors by ${\mathbb {DQ}}^n$.  For $\vx = (x_1, x_2,\cdots, x_n)^\top \in {\mathbb {Q}}^n$ with $x_i = (x_i)_0 + (x_i)_1\ii + (x_i)_2\jj + (x_i)_3\kk$ for $(x_i)_j \in \mathbb R$, $i = 1, 2,\cdots, n, j = 0, 1, 2, 4$, denote
$$\vx^R = ((x_1)_0, (x_1)_1, (x_1)_2, (x_1)_3, (x_2)_0, \cdots, (x_n)_3)^\top \in \mathbb {R}^{4n}.$$
Then
$$\|\vx\|_2 \equiv \sqrt{\sum_{i=1}^n |x_i|^2} = \sqrt{\sum_{i=1}^n \left[(x_i)_0^2 + (x_i)_1^2 + (x_i)_2^2 +(x_i)_3^2\right]} \equiv \|\vx^R\|_2.$$
This connects the $2$-norms of real vectors and quaternion vectors, and will be useful for us to discuss the $2$-norm of dual quaternion vectors.

For $\vx = (x_1, x_2,\cdots, x_n)^\top, \vy = (y_1, y_2,\cdots, y_n)^\top  \in {\mathbb {Q}}^n$, define $\vx^*\vy = \sum_{i=1}^n x_i^*y_i$, where $\vx^* = (x_1^*, x_2^*,\cdots, x_n^*)$ is the conjugate transpose of $\vx$.
By Theorem \ref{t4.2}, we have the following proposition.

\begin{Prop} \label{p6.1}
For $\vx, \vy \in {\mathbb {Q}}^n$, we have $\vx^* \vy + \vy^*\vx \le 2\|\vx\|_2\|\vy\|_2$.
\end{Prop}

\begin{proof}
By Theorem \ref{t4.2}, we know $$x_i^*y_i+y_i^*x_i= 2[(x_i)_0(y_i)_0 + (x_i)_1(y_i)_1+(x_i)_2(y_i)_2+(x_i)_3(y_i)_3].$$
Consequently, it holds that
$$
\vx^* \vy + \vy^*\vx =\sum_{i=1}^n(x_i^*y_i+y_i^*x_i)=2(\vx^R)^\top\vy^R  \le 2\|\vx^R\|_2\|\vy^R\|_2
= 2\|\vx\|_2\|\vy\|_2.$$
\end{proof}

For $\vx \in
{\mathbb {DQ}}^n$, we may also write
$$\vx = \vx_{st} + \vx_\I\epsilon,$$
where $\vx_{st}, \vx_\I \in \mathbb {Q}^n$ are the standard part and the infinitesimal part of $\vx$ respectively.

A function $v: {\mathbb {DQ}}^n \to \mathbb D$ is called a norm on ${\mathbb {DQ}}^n$ if it satisfies the following three properties:

1. For any $\vx \in {\mathbb {DQ}}^n$, $v(\vx) \ge 0$, and $v(\vx) = 0$ if and only if $\vx = \0$;

2. For any $\vx \in {\mathbb {DQ}}^n$ and $q \in \mathbb {DQ}$, $v(q\vx) = |q|v(\vx)$;

3. For any $\vx, \vy \in {\mathbb {DQ}}^n$, $v(\vx+\vy) \le v(\vx)+v(\vy)$.

Suppose that $\vx = (x_1, x_2,\cdots, x_n)^\top$.   We may extend the $1$-norm and $\infty$-norm to dual quaternion vectors as follows:
\begin{equation} \label{e10}
\|\vx\|_1 = \sum_{i=1}^n |x_i|,
\end{equation}
and
\begin{equation} \label{e11}
\|\vx\|_\infty = \max_{i=1,2, \cdots, n} |x_i|.
\end{equation}

For $i=1, 2,\cdots, n$, we have
\begin{equation} \label{e14}
|x_i| = \left\{ \begin{aligned} |(x_i)_{st}| + {(x_i)_{st}(x_i)_\I^*+(x_i)_\I (x_i)_{st}^* \over 2|(x_i)_{st}|}\epsilon, & \ {\rm if}\  (x_i)_{st} \not = 0, \\
|(x_i)_\I|\epsilon, &  \ {\rm otherwise}.
\end{aligned} \right.
\end{equation}

\begin{Prop} \label{p6.2}
The $1$-norm and the $\infty$-norm, defined above satisfy the three properties of norms.
\end{Prop}
\begin{proof}

Consider the $1$-norm first. Let $\vx \in {\mathbb {DQ}}^n$.  If $\vx = \0$, then $x_i = 0$ for $i=1, 2,\cdots, n$.   By (\ref{e14}), we have $|x_i| = 0$ for $i=1, 2,\cdots, n$.  By (\ref{e10}), we have $\|\vx\|_1 = 0$. On the other, if $\|\vx\|_1 = 0$, then by (\ref{e10}), we have $|x_i| = 0$ for $i=1, 2,\cdots, n$, as $|x_i| \ge 0$ for $i=1, 2,\cdots, n$. By Theorem \ref{t5.1}, we have $x_i = 0$ for $i=1, 2,\cdots, n$.  Hence $\vx = \0$.  This proves Property 1 for the $1$-norm.  Now, let $q \in \mathbb {DQ}$.   We have
$$\|q\vx\|_1 = \sum_{i=1}^n |qx_i| = \sum_{i=1}^n |q| |x_i| = |q| \sum_{i=1}^n |x_i| = |q|\|\vx\|_1.$$
Then, Property 2 of the norm holds for the $1$-norm.  Finally, for $\vx, \vy \in {\mathbb {DQ}}^n$, we have
$$\|\vx+\vy\|_1 = \sum_{i=1}^n |x_i+y_i| \le \sum_{i=1}^n (|x_i|+|y_i|) = \sum_{i=1}^n |x_i| + \sum_{i=1}^n |y_i| = \|\vx\|_1+\|\vy\|_1.$$
This proves Property 3 for the $1$-norm.  Hence, the $1$-norm is a norm.

We see that the proof for the $1$-norm is the same as in the real vector space.  The proof for the $\infty$-norm is also the same as in the real vector space.   Hence, we omit the details here.
\end{proof}

However, for $2$-norm, we may not simply define
\begin{equation} \label{e12}
\|\vx\|_2 = \sqrt{\sum_{i=1}^n |x_i|^2}.
\end{equation}
We should define $\|\vx\|_2$ by (\ref{e12}) if not all of $x_i$ are infinitesimal.   If all $x_i$ are infinitesimal, we have $x_i = (x_i)_\I \epsilon$ for $i = 1, 2,\cdots, n$.  Then we define
\begin{equation} \label{e13}
\|\vx\|_2 = \sqrt{\sum_{i=1}^n |(x_i)_\I |^2} \epsilon.
\end{equation}

\begin{Prop} \label{p6.3}
For any $\vx =\vx_{st} + \vx_\I\epsilon\in{\mathbb {DQ}}^n$ with $\vx_{st}\neq \0$, it holds that
\begin{equation}\label{e12-1}
\|\vx\|_2=\|\vx_{st}\|_2+\frac{(\vx_{st}^R)^\top \vx_\I^R }{\|\vx_{st}\|_2}\epsilon\leq \|\vx_{st}\|_2+\|\vx_\I\|_2\epsilon.
\end{equation}
\end{Prop}

\begin{proof}
Since $\vx_{st}\neq \0$, by (\ref{e14}) and (\ref{e12}), we have
$$
\begin{array}{lll}
\|\vx\|_2&=&\displaystyle\sqrt{\sum_{i=1}^n|x_i|^2}\\
&=&\displaystyle\sqrt{\sum_{(x_i)_{st}\neq0}\left(|(x_i)_{st}|+\frac{(x_i)_{st}(x_i)_\I^*+(x_i)_\I (x_i)_{st}^*}{2|(x_i)_{st}|}\epsilon\right)^2+\sum_{(x_i)_{st}=0}|(x_i)_\I\epsilon|^2}\\
&=&\displaystyle\sqrt{\sum_{(x_i)_{st}\neq0}\left(|(x_i)_{st}|^2+((x_i)_{st}(x_i)_\I^*+(x_i)_\I (x_i)_{st}^*)\epsilon\right)}.
\end{array}
$$
Since $(x_i)_{st}^*=0$ when $(x_i)_{st}=0$, we further have
$$
\begin{array}{lll}
\|\vx\|_2&=&\displaystyle\sqrt{\sum_{i=1}^n|(x_i)_{st}|^2+\left(\sum_{i=1}^n((x_i)_{st}(x_i)_\I^*+(x_i)_\I (x_i)_{st}^*)\right)\epsilon}\\
&=&\displaystyle\sqrt{\|\vx_{st}\|_2^2+\left(\sum_{i=1}^n((x_i)_{st}(x_i)_\I^*+(x_i)_\I (x_i)_{st}^*)\right)\epsilon}.
\end{array}
$$
Consequently, by (\ref{e4}) and Theorem \ref{t4.2}, we have
$$
\begin{array}{lll}
\|\vx\|_2&=&\|\vx_{st}\|_2+\displaystyle\frac{\sum_{i=1}^n((x_i)_{st}(x_i)_\I^*+(x_i)_\I (x_i)_{st}^*)}{2\|\vx_{st}\|_2}\epsilon\\
&=&\|\vx_{st}\|_2+\displaystyle\frac{\sum_{i=1}^n((x_i)_{st}^*(x_i)_\I+(x_i)_\I^* (x_i)_{st})}{2\|\vx_{st}\|_2}\epsilon\\
&=&\|\vx_{st}\|_2+\displaystyle\frac{\vx_{st}^*\vx_\I+\vx_\I^*\vx_{st}}{2\|\vx_{st}\|_2}\epsilon\\
&=&\|\vx_{st}\|_2+\displaystyle\frac{(\vx_{st}^R)^\top\vx_\I^R}{\|\vx_{st}\|_2}\epsilon,
\end{array}
$$
which means that the equality in (\ref{e12-1}) holds. Finally, the inequality in (\ref{e12-1}) follows from the fact that $(\vx_{st}^R)^\top\vx_\I^R\leq \|\vx_{st}\|_2\|\vx_\I\|_2$.
\end{proof}

\begin{Thm} \label{t6.3}
The $2$-norm, defined by (\ref{e12}) and (\ref{e13}), satisfies the three properties of norms.
\end{Thm}
\begin{proof} By (\ref{e12}) and (\ref{e13}), if $\vx = \0$, then $\|\vx\|_2 = 0$.  On the other hand, assume that $\|\vx\|_2 = 0$.  If one of $x_i$ is appreciable, by (\ref{e12}), we have
$|x_i| = 0$ for $i=1, 2,\cdots, n$.   This implies that $x_i = 0$ for $i=1, 2,\cdots, n$.  Hence $\vx = \0$.   If all of $x_i$ are infinitesimal, i.e., $\vx_{st}=\0$, then by (\ref{e13}), we have
$|(x_i)_\I | = 0$ for $i=1, 2,\cdots, n$.  By Proposition \ref{p4.1}, this means that $(x_i)_\I  = 0$ for $i=1, 2,\cdots, n$.   Hence, $x_i = 0$ for $i=1, 2,\cdots, n$, i.e., $\vx = \0$.   Thus, Property 1 holds for the $2$-norm.

Assume that $q \in \mathbb {DQ}$ is appreciable.  If at least one $x_i$ is appreciable, then the corresponding $qx_i$ is also appreciable.   By (\ref{e12}),
$$\|q\vx\|_2 = \sqrt{\sum_{i=1}^n |qx_i|^2}.$$
By Theorem \ref{t5.1}, $|qx_i| = |q| |x_i|$.   Hence,
$$\|q\vx\|_2 =\sqrt{|q|^2\sum_{i=1}^n |x_i|^2} = \sqrt{|q|^2\|\vx\|_2^2}.$$
By Theorem \ref{t5.1}, $|x_i| \ge 0$ for all $i$.   Then $|x_i|^2 \ge 0$ for all $i$.  By (\ref{e7}), $|x_i|$, hence $|x_i|^2$ is appreciable, if $x_i$ is appreciable.  This implies that $\sum_{i=1}^n |x_i|^2 \ge 0$ and is appreciable.   Since $q$ is appreciable, by (\ref{e7}), $|q| \ge 0$ and is also appreciable.
Thus, $|q|^2 \ge 0$ and is appreciable.  Then, $|q|^2\|\vx\|_2^2 \ge 0$ and is also appreciable.  We have
$$\|q\vx\|_2 = \sqrt{|q|^2\|\vx\|_2^2} = |q|\|\vx\|_2.$$
If all $x_i$ are infinitesimal, then all $qx_i$ are also infinitesimal.   By (\ref{e13}), we have
$$\|q\vx\|_2  =  \sqrt{\sum_{i=1}^n |q_{st}(x_i)_\I |^2} \epsilon
 =  |q_{st}|\sqrt{\sum_{i=1}^n |(x_i)_\I |^2} \epsilon.$$
 Since $q_{st} \not = 0$, by (\ref{e7}), we have
 $$|q_{st}| = |q| - {q_{st}q_\I^*+q_\I q_st^* \over 2|q_{st}|}\epsilon.$$
 $$\|q\vx\|_2  = \left(|q| - {q_{st}q_\I^*+q_\I q_st^* \over 2|q_{st}|}\epsilon\right)\sqrt{\sum_{i=1}^n |(x_i)_\I |^2} \epsilon =  |q|\sqrt{\sum_{i=1}^n |(x_i)_\I |^2}\epsilon
 =  |q|\|\vx\|_2.$$
Assume now that $q$ is infinitesimal.  Then $q = q_\I \epsilon$, which implies that all $qx_i=q_\I (x_i)_{st}\epsilon$ are infinitesimal. Consequently, by (\ref{e13}), we have
\begin{equation} \label{e15}
\|q\vx\|_2  =  \sqrt{\sum_{i=1}^n |q_\I (x_i)_{st} |^2} \epsilon = |q_\I | \sqrt{\sum_{i=1}^n |(x_i)_{st} |^2} \epsilon=|q_\I | \|\vx_{st}\|_2 \epsilon.
\end{equation}
If $\vx_{st}\neq\0$, then by (\ref{e15}) and Proposition \ref{p6.3}, we have
$$\|q\vx\|_2  =  |q_\I | \left(\|\vx\|_2-\frac{(\vx_{st}^R)^\top \vx_\I^R }{\|\vx_{st}\|_2}\epsilon
\right)\epsilon=|q_\I | \|\vx\|_2\epsilon= |q|\|\vx\|_2.
$$
If $\vx_{st}=\0$, then $q\vx= \0$, which implies $\|q\vx\|_2 = 0$ by (\ref{e13}).  But in this case, by (\ref{e13}), $\|\vx\|_2$ is also infinitesimal, which implies, together with the fact that $|q| = |q_\I |\epsilon$, that $|q|\|\vx\|_2 = 0 $. Hence $\|q\vx\|_2 =|q|\|\vx\|_2$. This proves that Property 2 holds for the $2$-norm.

Finally, let $\vx= \vx_{st} + \vx_\I \epsilon, \vy = \vy_{st} + \vy_\I \epsilon\in {\mathbb {DQ}}^n$.   We wish to prove that
\begin{equation} \label{e16}
\|\vx+\vy\|_2 \le \|\vx\|_2 + \|\vy\|_2.
\end{equation}
By the properties of the $2$-norm for quaternions \cite{WLZZ18}, we have
\begin{equation} \label{e17}
\|\vx_{st}+\vy_{st}\|_2 \le \|\vx_{st}\|_2 + \|\vy_{st}\|_2.
\end{equation}
If
\begin{equation} \label{e18}
\|\vx_{st}+\vy_{st}\|_2 < \|\vx_{st}\|_2 + \|\vy_{st}\|_2,
\end{equation}
then $\vx_{st}\neq\0$ and $\vy_{st}\neq\0$. If $\vx_{st}+\vy_{st}\neq \0$, then by Proposition \ref{p6.3}, we have
$$\|\vx+\vy\|_2= \|\vx_{st}+\vy_{st}\|_2+ u\epsilon,$$
where $u =(\vx_{st}^R+\vy_{st}^R)^\top (\vx_\I^R+\vy_\I^R)/\|\vx_{st}+\vy_{st}\|_2$. Since $\vx_{st}\neq\0$ and $\vy_{st}\neq\0$, by Proposition \ref{p6.3}, we have
$$
\|\vx\|_2+\|\vy\|_2= \|\vx_{st}\|_2+\|\vy_{st}\|_2+\left(\frac{(\vx_{st}^R)^\top\vx_\I^R}{\|\vx_{st}\|_2}+\frac{(\vy_{st}^R)^\top\vy_\I^R}{\|\vy_{st}\|_2}\right)\epsilon.
$$
By (\ref{e18}), we know 
that (\ref{e16}) holds.  If $\vx_{st}+\vy_{st}=\0$, then $\|\vx_{st}+\vy_{st}\|_2 = 0 < \|\vx_{st}\|_2+ \|\vy_{st}\|_2$. Thus, (\ref{e16}) still holds.

If $\|\vx_{st}+\vy_{st}\|_2 = \|\vx_{st}\|_2+\|\vy_{st}\|_2$, 
then by the argument at the beginning of this section, we know that
$$\|\vx_{st}^R+\vy_{st}^R\|_2 = \|\vx_{st}^R\|_2+\|\vy_{st}^R\|_2.$$
By the properties of the $2$-norm of real vectors, either $\vx_{st} = \0$ or $\vy_{st}=\0$ or there is a real positive number $t$ such that $\vy_{st} = t\vx_{st}$.    Hence, we divide this case to four subcases.

a. $\vx_{st} = \vy_{st} = \0$.   Then $\vx = \vx_\I \epsilon$ and $\vy = \vy_\I \epsilon$.  We have $\vx + \vy = (\vx_\I + \vy_\I )\epsilon$, and by (\ref{e13}),
$$\|\vx+\vy\|_2 = \|\vx_\I + \vy_\I \|_2\epsilon \le \left(\|\vx_\I \|_2 + \|\vy_\I \|_2 \right)\epsilon = \|\vx_\I \|_2 \epsilon + \|\vy_\I \|_2 \epsilon = \|\vx\|_2 + \|\vy\|_2.$$

b. $\vx_{st} = \0$ and $\vy_{st} \not = \0$.  Then $\vx + \vy = \vy_{st} + (\vx_\I +\vy_\I) \epsilon$. Since $\vy_{st} \not = \0$, by Proposition \ref{p6.3}, we have $\|\vx+\vy\|_2 = \|\vy_{st}\|_2 +v\epsilon$, where
$$v = \frac{(\vy_{st}^R)^\top (\vx_\I^R+\vy_\I^R) }{\|\vy_{st}\|_2}=\frac{(\vy_{st}^R)^\top \vx_\I^R}{\|\vy_{st}\|_2}+\frac{(\vy_{st}^R)^\top \vy_\I^R}{\|\vy_{st}\|_2}.$$
Consequently, we have
\begin{eqnarray*}
\|\vx+\vy\|_2 &=& \|\vy_{st}\|_2 +\frac{(\vy_{st}^R)^\top \vx_\I^R}{\|\vy_{st}\|_2}\epsilon+\frac{(\vy_{st}^R)^\top \vy_\I^R}{\|\vy_{st}\|_2}\epsilon\\
&=&\|\vy\|_2 +\frac{(\vy_{st}^R)^\top \vx_\I^R}{\|\vy_{st}\|_2}\epsilon\\
&\leq&\|\vy\|_2 +\|\vx_\I\|_2\epsilon\\
&=&\|\vy\|_2 +\|\vx\|_2,
\end{eqnarray*}
where the last second inequality is due to the fact that $(\vy_{st}^R)^\top \vx_\I^R\leq \|\vy_{st}^R\|_2\|\vx_I^R\|_2=\|\vy_{st}\|_2\|\vx_I\|_2$.
Thus, we have (\ref{e16}).

c. $\vx_{st} \not = \0$ and $\vy_{st} = \0$.   By exchanging $\vx$ and $\vy$ in the subcase b, we also have (\ref{e16}).

d. $\vx_{st} \not = \0$ and $\vy_{st} = t\vx_{st}$ for a real positive number $t$. In this case, since $\vx + \vy = (1+t)\vx_{st} + (\vx_\I + \vy_\I )\epsilon$, by Proposition \ref{p6.3}, we have
\begin{eqnarray*}
\|\vx + \vy\|_2= (1+t)\|\vx_{st}\|_2 +\frac{(\vx_{st}^R)^\top(\vx_\I^R + \vy_\I^R )}{\|\vx_{st}\|_2}\epsilon=\|\vx_{st}\|_2 +\|\vy_{st}\|_2+ \frac{(\vx_{st}^R)^\top\vx_\I^R}{\|\vx_{st}\|_2}\epsilon+\frac{(\vy_{st}^R)^\top\vy_\I^R }{\|\vy_{st}\|_2}\epsilon,
\end{eqnarray*}
where the second equality comes from $\vy_{st} = t\vx_{st}$. By Proposition \ref{p6.3} again, we know that (\ref{e16}) holds. Thus, Property 3 holds for the 2-norm.
\end{proof}

For any $\vx \in {\mathbb {DQ}}^n$, it is not difficult to show that $\|\vx\|_\infty \le \|\vx\|_2 \le \|\vx\|_1$.

For $\vx = (x_1, x_2,\cdots, x_n)^\top, \vy = (y_1, y_2,\cdots, y_n)^\top  \in {\mathbb {DQ}}^n$, let the conjugate transpose of $\vx$ be $\vx^* = (x_1^*, x_2^*,\cdots, x_n^*)$, and define $\vx^*\vy = \sum_{i=1}^n x_i^*y_i$.  If $\vx^*\vy = 0$, then we say that $\vx$ and $\vy$ are orthogonal.  It is not difficult to show that $\vx^*\vx = 1$ if and only if $\|\vx\|_2 =1$.  In this case, we say that $\vx$ is a unit dual quaternion vector.  If $\vx^{(1)}, \vx^{(2)},\cdots, \vx^{(n)} \in {\mathbb {DQ}}^n$, and $\left(\vx^{(i)}\right)^*\vx^{(j)} = \delta_{ij}$ for $i, j = 1, 2,\cdots, n$, where $\delta_{ij}$ is the Kronecker symbol, then we say that $\left\{\vx^{(1)}, \vx^{(2)},\cdots, \vx^{(n)}\right\}$ is an orthonormal basis of ${\mathbb {DQ}}^n$.

\section{Final Remarks}

In the study of robotics, dual quaternion optimization problems are studied \cite{BK20, CKJC16}.   In such
dual quaternion optimization problems, are the variables of the functions involved dual quaternion vectors?   Are those functions real valued or dual number valued?   How to analyze such optimization problems and their algorithms?   A further study is needed to address these problems.

The further study and applications of dual numbers, dual complex numbers and dual quaternions inevitably lead to the study on dual number matrices, dual complex matrices, dual quaternion matrices and their spectral theories \cite{Br20, Gu21, QL21}.   In particular, recently, Gutin \cite{Gu21} studied spectral theory and singular value decomposition of dual number matrices, Qi and Luo \cite{QL21} studied  spectral theory and singular value decomposition of dual complex matrices.   What about the spectral theory of dual quaternion matrices?  This may be also worth further studying.

\bigskip




\end{document}